\documentclass{elsart}

\usepackage[utf8]{inputenc}
\usepackage{amsmath}
\usepackage{amsfonts}
\usepackage{amssymb}
\usepackage{graphicx}
\usepackage{cite}

\newtheorem{definition}{\bf Definición}[section]
\newtheorem{Lem}[definition]{\bf Lemma}
\newtheorem{Thm}[definition]{\bf Theorem}
\newtheorem{Prop}[definition]{\bf Proposition}

\newtheorem{facts}{\bf Facts}
\newenvironment{proof}{\noindent \textit{Proof. }}{\qed}

\renewcommand{\d}{\delta}
\newcommand{\g}{\gamma}
\newcommand{\G}{\Gamma}

\begin{document}

\begin{frontmatter}

\title{On global offensive alliance in zero-divisor graphs}

\author{Ra\'ul Ju\'arez Morales$^1$}
\address{Escuela Superior de Ciencias y Tecnologías de la Información, Universidad Autónoma de Guerrero, Las Colinas 37A, Col. Fracc. Las playas, Acapulco, Guerrero, M\'exico.} \ead{raul.juarezm@uagro.mx}

\author{Gerardo Reyna$^{2}$}
\address{Facultad de Matemáticas, Universidad Autónoma de Guerrero,
Carlos E. Adame 5, Col. La Garita, Acapulco, Guerrero, M\'exico} \ead{gerardoreynah@hotmail.com}

\author{Jesús Romero Valencia $^4$}
\address{Facultad de Matemáticas, Universidad Autónoma de Guerrero,
Carlos E. Adame 5, Col. La Garita, Acapulco, Guerrero, M\'exico} \ead{jromv@yahoo.com}

\author{Omar Rosario$^3$}
\address{Facultad de Matemáticas, Universidad Autónoma de Guerrero,
Carlos E. Adame 5, Col. La Garita, Acapulco, Guerrero, M\'exico} \ead{omarrosarioc@gmail.com}


\date{\today}

\begin{abstract}
 Let $\G(V,E)$ be a simple graph without loops nor multiple edges. A nonempty subset $S \subseteq V$ is said a {\em global offensive alliance} if  every vertex $v \in V- S$  satisfies that $\d_S(v) \geq \d_{\overline{S}}(v)+1$. The {\em global offensive alliance number} $\g^o(\Gamma)$ is defined as the minimum cardinality among all global offensive alliances. Let $R$ be a finite commutative ring with identity.  In this paper, we initiate the study of the global offensive alliance number of the zero-divisor graph $\G(R)$.
\end{abstract}


\begin{keyword}
Offensive alliances; Zero-divisor graph.
\end{keyword}
\end{frontmatter}

\section{Introduction}
In this paper $\G = (V,E)$ denotes a simple graph of order $n$ and size $m$ where $V$ is the  vertex set and $E$ the edge set. The degree of a vertex $v\in V$ will be denoted by  $\d(v)$,  the minimum degree will be denoted by $\d$, and the maximum degree by $\Delta$. The subgraph induced by a set $S\subset V$ will be denoted by $\langle S \rangle$. For a non-empty subset $S\subset V$, and a vertex $v\in V$ the set of neighbors that $v$ has in $S$ is $N_{S}(v) := \{ u\in S : u\sim v\}$. We will designate $\overline{S} = V-S$.

The mathematical properties of alliances in graphs were first studied by Kristiansen {\em et al}. \cite{KrHeHe}. They proposed alliances of different types, such as defensive alliances \cite{RodYerSig, RodSig, SigRod, YerRod, YerBerRodSig}, offensive alliances \cite{FavFriGod, Rad, BRSY} and dual alliances (also known as powerful alliances) \cite{BriDutHay}. These types of alliances have been extensively studied in the last decade.  A generalization of these alliances called {\em $k$-alliances} were introduced by Shafique and Dutton \cite{ShaDut}. We are interested in the study of the mathematical properties of global offensive alliances. Let $S$ be a nonempty subset of the vertex set $V$, $S$ is a {\em global offensive alliance} of $\G = (V,E)$ if satisfies that 
\begin{equation}\label{aog}
\d_S(v) \geq \d_{\overline{S}}(v)+1 \hspace{1cm} for \hspace{0.1cm} all \hspace{0.1cm} v \in \overline{S}.
\end{equation}
If for a vertex $v \in \overline{S}$ the relation \ref{aog} is verified, we will say that {\em $v$ satisfies the global offensive alliance condition in S}. The {\em global offensive alliance number} $\g^o(\Gamma)$ is defined as the minimum cardinality among all global offensive alliances. For convenience, we will call $\gamma^o -$\emph{alliance} to a global offensive alliance of minimum cardinal.

In this work, $R$ will denote a finite commutative ring with identity. The {\em zero-divisor graph} is the simple graph $\Gamma(R)$, with vertex set the proper zero-divisors of $R$, {\em i.e.}, $Z(R)^{*} = Z(R) - \{0\}$, and for different $u$, $v$ $\in Z(R)^{*}$, the vertices $u$ and $v$ are adjacent if and only if $uv=0$. If $Ann(v)$ denotes the annihilator of $v$ (that is, the set of elements $u\in R$ such that $uv = 0$) notice that $\d(v)=\vert Ann(v) \vert -1$ if $v^2 \neq 0$ or $\d(v)=\vert Ann(v) \vert - 2$ otherwise. Beck was the first to introduce the concept of zero-divisor graph in 1988 for his study of the coloring of a commutative ring \cite{Beck}. Beck considered all the elements of the ring as vertices of the graph and an edge is obtained if two different elements $u$ and $v$ satisfy that $uv=0$. In \cite{AndLiv}, Anderson and Livingston introduced and studied the zero-divisor graph with a slight  modification, the vertices of the graph are the proper zero-divisors. In \cite{MutMam}, Muthana and Mamouni initiated the study of the global defensive alliance number of the zero-divisor graphs. They also considered as vertices of the graph  the proper zero-divisors.

In this paper, we initiate the study of the global offensive alliance number of the zero-divisor graph $\Gamma(R)$ where $R$ is a finite commutative ring with identity, we will use as vertices of the graph the proper zero-divisors. In Section \ref{sec2} we give some results concerning the global offensive alliance number of the zero-divisor graph, for example, we give a characterization in terms of the global offensive alliance number in order to $\Gamma(R)$ be a complete graph, also we give sharp bounds of $\gamma^o(\Gamma(R))$ for different kinds of rings $R$ ((co)local and certain direct products of rings). Finally, in Section \ref{sec3} we give a complete characterization of rings with $\gamma^o(\Gamma(R)) = 1$, or $2$.

\section{The Global Offensive Alliance Number of the Zero-Divisor Graph} \label{sec2}

In this section, our main goal is calculate or provide sharp bounds of the global offensive alliance number of zero-divisor graphs for some kind of direct products of finite local rings with finite fields. In particular, characterize when the global offensive alliance number over the ring formed by the direct product of $\mathbb{Z}_2$ with any ring $R$ is $|Z(R)^\ast| + 1$.

Recall that a commutative ring $R$ is a \emph{local ring} if it contains a unique maximal ideal  $M$. Throughout this paper we will use freely the following well known facts concerning this sort of rings (see \cite{MutMam}).

\begin{facts}\label{fac1}
Let $(R,M)$ be a finite local ring which is not a field. Then, the following holds.
\begin{description}
\item[(1)] $M=Z(R)=Ann_R(x)$ for some $x \in Z(R)^*$.
\item [(2)] $|R|=p^{nr}$, and $|M|=p^{(n-1)r}$ for some prime integer $p$, and some positive integers $n$ and $r$.
\end{description}
\end{facts}

\begin{Thm}
Let $R$ be a ring which is not a field, then $$ \gamma^o(\Gamma(R))+2 \leq \vert Z(R) \vert.$$
Moreover, this bound is sharp. 
\end{Thm}

\begin{proof}
Let $k=\gamma^o(\Gamma(R))$ and $S\subseteq V(\Gamma(R))$ be a global offensive alliance with $\vert S\vert = k$. Since $S$ is a minimum global offensive alliance $S\subseteq N(\overline{S})$ therefore 

$$
\begin{array}{ccl}
k & \leq & \vert N(\bar{S}) \vert =  \left\vert \displaystyle\bigcup_{v\in\bar{S}}{N(v)} \right\vert \leq \left\vert \displaystyle\bigcup_{v\in\bar{S}}{Ann(v)} \right\vert \\ 
 & = & \left\vert \left[S \bigcap \left(\displaystyle\bigcup_{v\in\bar{S}}{Ann(v)}\right)\right] \cup \left[\bar{S} \bigcap \left(\displaystyle\bigcup_{v\in\bar{S}}{Ann(v)}\right)\right] \right\vert \\ 
 & = & \left\vert S \bigcap \left(\displaystyle\bigcup_{v\in\bar{S}}{Ann(v)}\right)\right\vert + \left\vert \bar{S} \bigcap \left(\displaystyle\bigcup_{v\in\bar{S}}{Ann(v)}\right) \right\vert \\ 
 & = & \left\vert  \displaystyle\bigcup_{v\in\bar{S}}{\left( S \cap Ann(v) \right)}\right\vert + \left\vert \displaystyle\bigcup_{v\in\bar{S}}{\left( \bar{S} \cap Ann(v) \right)} \right\vert \\ 
 & \leq & \displaystyle\sum_{v\in\bar{S}}{\left\vert S \cap Ann(v) \right\vert} + \displaystyle\sum_{v\in\bar{S}}{\left\vert \bar{S} \cap Ann(v) \right\vert} \\ 
 & \leq & \displaystyle\sum_{v\in\bar{S}}{d_S{(v)}} + \displaystyle\sum_{v\in\bar{S}}{(d_{\bar{S}}{(v)} + 1)}  \\ 
 & \leq & \displaystyle 2\sum_{v\in\bar{S}}{k} \\ 
 & = & 2k(\left\vert Z(R)^* \right\vert -k).
\end{array} 
$$

Hence $k \leq 2k(\left\vert Z(R)^* \right\vert -k).$ In other words $\left\vert Z(R)^* \right\vert \geq k + \frac{1}{2}$ therefore $\left\vert Z(R) \right\vert \geq k+2.$

In order to see this bound is sharp, take $R= \mathbb{Z}_9$. A straightforward computation shows that $ \vert Z(R) \vert = 3$ and $\g^o(\G(R))=1$.
\end{proof}

In \cite{MutMam} is shown that if $R$ is a finite commutative ring which is not a field then $\vert Z(R) \vert \leq \g_a(\G(R))^2 + \g_a(\G(R)) +1$ (where $\g_a(\G(R))$ denotes the global defensive alliance number) and equality holds also with $R=\mathbb{Z}_9$. Hence next result follows at once.

\begin{cor}
If $R$ is a ring which is not a field, then $$\g^o(\G(R))+ 2 \leq \vert Z(R) \vert  \leq \g_a(\G(R))^2 + \g_a(\G(R)) +1$$ and the two bounds are sharp.
\end{cor} 

\begin{Thm}
If $R$ is a local ring which is not a field, then $$\g^o(\G(R)) \leq \left(\vert Z(R)^* | - \g^o(\G(R))\right)( \g^o(\G(R)) -1) + 1.$$
\end{Thm}

\begin{proof}
Let $k= \g^o(\G(R))$, if any $\gamma^o-$alliance $S$, contains some element $ x \in Z(R)^*$ with $Z(R) = Ann(x)$, then 

$$
\begin{array}{ccl} 
\vert S \vert & = & 1 + \vert S - \{ x \} \vert = 1 + \left\vert (S - \{ x \}) \cap N(\overline{S}) \right\vert \\ 
 & = & 1 + \left\vert(S-\{x\})\cap \left (\displaystyle \bigcup_{v\in \overline{S}}N(v)\right) \right\vert \\ 
 & = & 1 + \left\vert \displaystyle\bigcup_{v\in \overline{S}} ((S-\{x\})\cap N(v))\right\vert \\ 
 & \leq & 1 + \displaystyle\sum_{v\in\bar{S}}({\vert S - \lbrace x \rbrace )\cap N(v) \vert}= 1+ \displaystyle\sum_{v\in\bar{S}}{\d_{S - \lbrace x \rbrace}(v)} \\ 
 & = & 1+ \displaystyle\sum_{v\in\bar{S}}{(\d_{S}(v)- 1)} \\ 
 & \leq & 1 + \displaystyle\sum_{v\in\bar{S}}{(k- 1)} \\ 
 & = & (|Z(R)^*|-k)(k-1) + 1.
\end{array} 
$$
In other case, if for each $x$ with $Ann(x) =Z(R)$ occurs that $x\in \overline{S}$, then $|\overline{S}| = |\overline{S} - \{ x\}|+1 = |Z(R)^*| - (k+1) \leq |Z(R)^*| - k +1\leq  (|Z(R)^*| - k )(k-1) + 1$.
\end{proof}

Notice that the inequality of preceding theorem  becomes an equality for some rings, for example  $R = \mathbb{Z}_9$.

In \cite[Remark 2.1]{BRSY} is shown that  $\g^o(K_{m,n}) = \min \{m,n\}$. If $F$ and $K$ are fields then $\G(F\times K) = K_{|F| -1, |K|-1}$ (see \cite[Example 3.4]{AndLiv}). Hence next lemma follows at once. 

\begin{Lem}\label{lema23}
If $F$ and $K$ are fields, then $$\g^o(\G(F\times K)) = \min\{|F| -1, |K|-1 \}. $$
\end{Lem}

\begin{Lem}
If $R$ is a ring, then
\begin{enumerate} \label{lema24}
\item [i)] The inequality $\g^o(\Gamma(\mathbb{Z}_2 \times R)) \leq |Z(R)|$ holds.
\item [ii)] If $R$ has at least two units, then $(1,0)$ is an element of every $\gamma^o-$alliance of $\Gamma(\mathbb{Z}_2\times R)$.
\end{enumerate}
\end{Lem}
\begin{proof}
\begin{enumerate}
\item [i)] The set $\{(1,0)\} \cup (\{0\} \times Z(R)^*)$ is a global offensive alliance.
\item [ii)] Let $S$ be a global offensive alliance that does not contain the element $(1,0)$. This implies that $\{0\}\times U(R)\subseteq S$. In this way $(S -(\{0\}\times U(R)))\cup \{(1,0)\}$ is a global offensive alliance with $|S| - |U(R)| + 1\leq |S| - 1$ vertices.
\end{enumerate}
\end{proof}

\begin{Thm}
If $R$ is a ring then $\Gamma(R)=K_n$ if and only if $\gamma^\circ(\Gamma(\mathbb{Z}_2\times R))=n+1$ and exactly one of the following conditions holds:  
\begin{enumerate}
\item [i)] $n\leq 4$ and $R \ncong $ $\mathbb{Z}_3\times \mathbb{Z}_3$, or 
\item[ii)] $ \d_{\Gamma(R)}\geq 4$.
\end{enumerate}
\end{Thm}
\begin{proof}
$\Rightarrow$) In \cite{AndLiv} can be found all zero-divisor graphs with at most four vertices, in any of those an easy computation shows the result and this establishes i).\\
For ii), if $n\geq 5$ by \cite[Theorem 2.8]{AndLiv} $R$ is a local ring such that for each $x\in Z(R)^*$, $Ann(x)=Z(R)$. \\
Let $S$ be a $\gamma^o-$alliance, if $x\in Z(R)^*$ is such that $(0,x)\notin S$, since the degree of this vertex in $\Gamma(\mathbb{Z}_2\times R)$ is  $2|Z(R)^*|$, it must occurs that $|S|\geq |Z(R)|=n+1$. The other inequality is just Lemma \ref{lema24}-i).\\
Suppose now that for all element $x\in Z(R)^*$, $(0,x)\in S$. Since $n\geq 5$,  Facts \ref{fac1} gives that  $R$ contains at least two units. According Lemma \ref{lema24} ii) it follows that  $(1,0) \in S$, hence $\{(1,0)\}\cup (\{0\}\times Z(R)^*)\subseteq S$ and again Lemma \ref{lema24}-i) gives $\gamma^o(\Gamma(\mathbb{Z}_2\times R))=n+1$.\\
$\Leftarrow$) i) If the zero-divisor graph $\Gamma(R)$ has at most four vertices and $R \ncong \mathbb{Z}_3 \times \mathbb{Z}_3$, once again an easy examination  at the  whole these zero-divisor graphs shows that $\Gamma(R)$ must be complete.\\
ii) The set $S=\{(0,1)\}\cup (\{0\}\times Z(R)^\ast)$ is a global offensive alliance which by hypothesis is in fact a $\gamma^o-$alliance. For each $x\in Z(R)^\ast$, since $S-\{(0,x)\}$ is not a global offensive  alliance and $\delta_{\Gamma(R)} \geq 4$ it follows that $x^2=0$.\\
The above has several implications: the {\em nilradical} of $R$ (i.e., the ideal of all those nilpotent elements $x \in R$) equals to $Z(R)$ and therefore $R$ is a local ring, each vertex of the form  $(0,x)$ (and of the form $(1,x)$) with $x\in Z(R)^\ast$ has degree at least $5$ and the set $X^\ast=\{x\in Z(R)^\ast| Ann(x)=Z(R)\}$ is nonempty.\\
The theorem is proved if we can see that $\overline{X^\ast}$  is empty. If this is not the case, there are different $x,y\in \overline{X^\ast}$ such that $xy\neq 0$. If $z\in X^\ast$, it can be verified that $(S-\{(0,x),(0,y)\})\cup\{(1,z)\}$ is a global offensive alliance with $n$ vertices, which is a contradiction.
\end{proof}

Recall that a ring $R$ is \emph{co-local} if it contains a non-trivial ideal contained in each non-trivial proper ideal.

\begin{Thm}\label{Rcolocal}
If a finite ring $R$ is co-local then is local.
\end{Thm}
\begin{proof}
If $R$ is a field the result is evident. If $R$ is not a field, let $N$ be the non-trivial ideal contained in each non-trivial proper ideal of $R$. Let  $x \in N$ be with $x \neq 0$, then $x$ is contained in every non-trivial ideal of $R$. In particular for each $y \in Z(R)^*$, $x \in Ann(y)$. Conversely, every $y \in Z(R)^*$ is such that $y \in Ann(x)$ in other words, $Ann(x)=Z(R)$ hence $R$ is local.
\end{proof}

\begin{Thm}
Let $R$ be a co-local ring  which is not a field and $N$ the non-trivial ideal contained in all proper non-trivial ideals of $R$. Then $\g^o(\G(R)) \geq \left \lceil \frac{\vert N^* \vert}{2} \right \rceil$. Furthermore this bound is sharp.
\end{Thm}
\begin{proof}
By Theorem \ref{Rcolocal}, $R$ is a local ring with unique maximal ideal $Z(R)\neq \{0\}$. Moreover $N \subseteq Z(R)$ and $\vert Z(R) \vert = p^n$ for some prime number $p$ and for some positive integer number $n$ and therefore $\vert N \vert = p^m$, with $m\leq n$.
If $N \varsubsetneq Z(R)$, then $\vert N \vert = p^m \leq p^{n-1} = \frac{p^n}{p} \leq \frac{p^n}{2} \leq \frac{\vert Z(R) \vert}{2} $ and from this follows that $\vert N^* \vert \leq \frac{\vert Z(R)^* \vert}{2}$. Now, if $S$ is a $\gamma^o -$alliance and some $x \in N^*$ does not belong to $S$, then $x$ is adjacent at least $\left \lceil \frac{\vert Z(R)^* \vert}{2} \right \rceil$ vertices in $S$.
On the other hand, if $N^* \subseteq S$, then $\vert S \vert \geq \vert N^* \vert \geq \left \lceil \frac{\vert N^* \vert} {2} \right \rceil$.\\
Finally, if $N=Z(R)$, then $\G(R)$ is a complete graph and it is well known that $\g^o(\G(R)) = \left \lceil \frac{\vert Z(R)^* \vert}{2} \right \rceil = \left \lceil \frac{\vert N^* \vert}{2} \right \rceil$.

It can be verified that the bound is attained for $R= \mathbb{Z}_9$.
\end{proof}

\begin{Thm}
Let $R$ be a ring and $r$ the minimum number of nilpotent elements of index $2$ contained in the complement of a $\gamma^o-$alliance, then 
$$\g^o(\G(\mathbb{Z}_2 \times R))\leq 1+ r + 2 \g^o(\G(R)).$$ 
Moreover the bound is sharp.
\end{Thm}
\begin{proof}
Let $S'$ be a $\gamma^o-$alliance of $\G(R)$ that satisfies the conditions of the statement and let $P=\{x_1, x_2, \ldots, x_r \}$ be the nilpotent elements of index $2$ contained in $\overline{S'}$. Set $S=\{ (1,0) \} \cup ( \{ 0 \} \times (S' \cup P)) $ and let us show that this set is a global offensive alliance of $\G(\mathbb{Z}_2 \times R)$. To this end, first observe that the elements of $\{0 \} \times U(R) $ satisfy the global offensive alliance condition. For an element $x \in \overline{S' \cup P}$, note that $\d_{\overline{S'}}(x) = \d_P(x) + \d_{\overline{S'} - P}(x)$ and $ \d_{\overline{S}}(0,x) = \d_{\overline{S'} - P}(x) + \d_{\overline{S'}}(x)$, therefore, 
$$    
\begin{array}{ccl}
\d_S(0,x) & = & 1 + 2\d_{S'}(x) +  \d_P(x) \\ 
 & \geq & 1 + 2(\d_{\overline{S'}}(x) + 1) + \d_P(x) \\ 
 & = & 3 + 2\d_{\overline{S'}}(x) + \d_P(x) \\ 
 & = & 3 + \d_{\overline{S'}}(x) + \d_P(x) + \d_{\overline{S'} - P}(x) + \d_P(x)\\ 
 & = & 3+\d_{\overline{S}}(0,x) + 2\d_P(x) \\
 & \geq & \d_{\overline{S}}(0,x) +1.
\end{array} $$

Moreover, since $\d_{\overline{S}}(1,x)= \d_{\overline{S' \cup P}}(x) \leq \d_{\overline{S'}}(x)$, we have
$$
\begin{array}{ccl}
\d_S(1,x) & = & \d_{S' \cup P}(x) \\ 
 & = & \d_{S'}(x) + \d_P(x) \\ 
 & \geq & \d_{\overline{S'}}(x) +1 + \d_P(x) \\ 
 & \geq  & \d_{\overline{S}}(1,x) +1 + \d_P(x) \\ 
 & \geq & \d_{\overline{S}}(1,x) +1.
\end{array} $$
A similar analysis to the previous one shows that if $x \in P$ then $\d_S(1,x) \geq \d_{\overline{S}}(1,x) +1$.

The bound is attained for $R = \mathbb{Z}_8$.
\end{proof}

\begin{Prop}
Let $R$ be a ring, then
\begin{enumerate}
  \item $\gamma^o(\Gamma(\mathbb{Z}_2\times R))\geq \gamma^o(\Gamma(R))+1$;
  \item and $\gamma^o(\Gamma(\mathbb{Z}_2\times R))\leq 2\gamma^o(\Gamma(R))+1$  if $R$ is a reduced ring (i.e., a ring not containing non-zero elements $x$ such that $x^2=0$).
\end{enumerate}
\end{Prop}
\begin{proof}
\emph{(1)} Let $S\subset Z(\mathbb{Z}_2\times R)^*$ be a dominating set, since $(1,0)$ is the unique neighbor of $(0,1)$, it must be $(1,0)\in S$ or $(0,1)\in S$. By cases.\\
Suppose $(1,0)\in S$ and consider $\pi_2:\mathbb{Z}_2\times Z(R)^*\rightarrow Z(R)^*$ the projection to the second factor, $S_1=\{(0,b):b\in Z(R)^*\}$ and $S_2=\{(1,b):b\in Z(R)^*\}$, then $|\pi_2(S_1)|+|\pi_2(S_2)|<\gamma^o(\Gamma(R))$, thus, $T=\pi_2(S_1)\cup \pi_2(S_2)$ is not a global offensive alliance of $\Gamma(R)$, which implies the existence of $y\in \overline{\pi_2(S_1)\cup \pi_2(S_2)}$ such that
$$\delta_{T}(y)<\delta_{\overline{T}}+1,$$
thus,
$$\delta_{S}(1,y)<\delta_{\overline{S}}(1,y)+1,$$
so $S$ is not a global offensive alliance of $\Gamma(\mathbb{Z}_2\times R)$.\\
Now, if $(0,1)\in S$ and $(1,0)\notin S$ we may observe that $|S|\ge |Z(R)^*|+|U(R)|$, then $\gamma^o(\Gamma(R))+1>|Z(R)^*|+|U(R)|$, thus, $\gamma^o(\Gamma(R))>|R^*|+1$ a contradiction.\\
Hence, $\gamma^o(\Gamma(\mathbb{Z}_2\times R))\geq \gamma^o(\Gamma(R))+1$.\\
For the second statement, observe that if $T\subset Z(R)^*$ is a minimal global offensive alliance, then $$S=\{(1,0)\}\cup (\{0\}\times T)\cup (\{1\}\times T)$$
is a global offensive alliance in $\Gamma(\mathbb{Z}_2\times R)$. Therefore, $$\gamma^o(\Gamma(\mathbb{Z}_2\times R))\leq 2\gamma^o(\Gamma(R))+1.$$
\end{proof}

\begin{Thm}
If $R$ is a ring such that $\Gamma(R)$ is a complete graph and $F$ a field with $|F|\geq 3$, then 
\small
$$\gamma^o(\Gamma(F \times R)) = |Z(R)^*| + \min \left\{ |U(R)|,|F^*|, 2 + \left\lfloor \frac{|U(R)|-|Z(R)^*|}{2} \right\rfloor +  \left\lfloor \frac{|F^*|}{2} \right\rfloor \right\}.$$
\end{Thm}
\normalsize
\begin{proof}
Let $A$ be the set on which the minimum is being considered. If $R= \mathbb{Z}_2 \times \mathbb{Z}_2$ then $\min A = |U(R)|=1$, $|Z(R)^*|=2$ and the set $S = \{ (0,0,1), (0,1,0), (0,1,1) \}$ is a $\gamma^o$-alliance.\\
If $R \neq \mathbb{Z}_2 \times \mathbb{Z}_2$ then $R$ is a local ring in which $xy=0$ for any $x,y \in Z(R)$.\\
If $R$ is a field, $|Z(R)^*|=0$, $\min A = \min\{|U(R)|, |F^*|\}$ and in this case, the result is simply the Lemma \ref{lema23}.\\
In other case, the proof is completed using following steps.

Step 1. \emph{ Any vertex of $\{0\}\times Z(R)^*$ must belong to each $\gamma^o$-alliance}.

Indeed, if $S \subseteq Z(F \times R)^* $ is a global offensive alliance and for some $x \in Z(R)^*$, $(0,x) \notin S$, then $S$ must contain at least $\left\lfloor \frac{|Z(R)^*|(|F^*|+1)+ |F^*|-1}{2} \right\rfloor +1$ of the neighbors of $(0,x)$, which are distributed among the sets $F^* \times \{0\}$, $\{0\} \times Z(R)^*$ and $F^* \times Z(R)^*$.\\
If $|S \cap (F^* \times Z(R)^*)| > |Z(R)^*|$ then $S-(S \cap (F^* \times Z(R)^*)) \cup (\{0\} \times Z(R)^*)$ is a global offensive alliance containing fewer vertices than $S$ does.
Now, if $|S \cap F^* \times Z(R)^*| \leq |Z(R)^*|$, some vertex of $F^* \times Z(R)^*$ not found in $S$ and  therefore it is must to have 
\begin{equation}\label{choco1}
|S \cap (\{ 0\} \times Z(R)^*)| \geq \left \lfloor \frac{|Z(R)^*|}{2} \right \rfloor + 1.
\end{equation}
And in order for the global offensive alliance condition to be satisfied at the vertex $(0,x)$ the following inequality must also be satisfied 
\begin{equation}\label{choco2}
|S \cap F^* \times Z(R)^*| \geq  \left \lfloor \frac{|Z(R)^*|}{2} \right \rfloor + 1 + |S \cap F^* \times \{ 0\}|.
 \end{equation}  
 
Thus, of inequalities \ref{choco1} and \ref{choco2} is obtained
 \small
 $$|(\{0\} \times Z(R)^*)-(S \cap (\{ 0\} \times Z(R)^*))| \leq \left \lceil \frac{|Z(R)^*|}{2} \right \rceil - 1 < \left \lfloor \frac{|Z(R)^*|}{2} \right \rfloor + 1 + |S \cap F^* \times \{ 0 \}|.$$
 \normalsize
Therefore $S_1=(S-(S\cap (F^* \times Z(R)^*))) \cup ((\{0\}\times Z(R)^*) - S) $ is a global offensive alliance and
$
|S_1|=|(S-(S\cap (F^* \times Z(R)^*)))| + |(\{0\}\times Z(R)^*) - S|=|S\cap (F^* \times \{0\})|+|Z(R)^*| = |S\cap (F^* \times \{0\})| + |S \cap (\{ 0\} \times Z(R)^*)| + |\{0\}\times Z(R)^* - S| < 2\left \lfloor \frac{|Z(R)^*|}{2} \right \rfloor + 2|S\cap (F^* \times \{0\})| +2 \leq |S|.$

\normalsize 

Step 2. \emph{The inequality $\gamma^o(\Gamma(F \times R)) \leq |Z(R)^*| + \min A$ holds}.

If $B \subseteq \{ 0 \} \times U(R)$ and $C \subseteq F^* \times \{ 0 \}$ are subset whose cardinal numbers are $\left\lfloor \frac{|U(R)|-|Z(R)^*|}{2} \right\rfloor + 1$ y $\left\lfloor \frac{|F^*|}{2} \right\rfloor + 1$ respectively, then the subsets $ \left( \{ 0 \} \times Z(R)^* \right)$ $\cup B \cup C$, $\left( \{ 0 \} \times Z(R)^* \right) \cup \left( F^* \times \{ 0 \} \right)$ y $\left( \{ 0 \} \times Z(R)^* \right) \cup \left( U(R) \times \{ 0 \} \right)$ are global offensive alliances, with what is certainly $\gamma^o(\Gamma(F \times R))\leq |Z(R)^*|+ \min A$.

Finally, if $S$ is a $\gamma^o-$alianza of $\Gamma(F \times R)$, by Step 1 we know that $S$ has the form $\{ 0 \} \times Z(R)^* \cup A_1$.\\
Now, if $A_1 \neq \{ 0 \} \times U(R)$, $F^* \times \{ 0 \}$ we have $(0,x) \in \{ 0 \} \times U(R) - S$ y $(y,0) \in F^* \times \{ 0 \}-S$. In order for the global offensive alliance condition to be met in these two vertices, the set $S$ is required to contain at least $\left\lfloor \frac{|U(R)|-|Z(R)^*|}{2} \right\rfloor +1$ vertices of $\{ 0 \} \times U(R)$ and at least $\left\lfloor \frac{|F^*|}{2} \right\rfloor +1$ vertices of $F^* \times \{ 0 \} $. In this way $$\gamma^o(\Gamma(F \times R)) \geq |Z(R)^*|+2+ \left\lfloor \frac{|U(R)|-|Z(R)^*|}{2} \right\rfloor + \left\lfloor \frac{|F^*|}{2} \right\rfloor \geq |Z(R)^*| + \min A.$$
\end{proof}

\begin{Thm}
Let $F$ and $K$ be two fields with $|F|\geq 3$, then $\gamma^o(\Gamma(\mathbb{Z}_2 \times K \times F)) = 1 + \min \left \{ 2|K^*|, 2|F^*|, |F^*| + \left \lfloor \frac{|K^*|}{2} \right \rfloor + 1, |K^*|+ \left \lfloor \frac{|F^*|}{2} \right \rfloor + 1 \right \}$.
\end{Thm}
\begin{proof}
Since $|\{0\} \times K^* \times F^*|\geq 2$ the vertex $(1,0,0)$ is in each $\gamma^o-$alliance. In addition, If $A$ is the set over which the minimum is under consideration, there are global offensive alliances of cardinal $1+a$ for each $a \in A$. For example
\begin{enumerate}
\item[i)] $ \{1,0,0\} \cup (\{0\} \times K^* \times \{0\}) \cup (\{1\} \times K^* \times \{0\})$,
\item[ii)] $ \{1,0,0\} \cup (\{0\} \times  \{0\} \times F^*) \cup (\{1\} \times  \{0\} \times F^* )$,
\item[iii)] $\{1,0,0\} \cup (\{0\} \times \{0\}\times F^*) \cup X $ (with $X \subseteq \{0\} \times K^* \times \{0\}$, $|X|=\lfloor \frac{|K^*|}{2} \rfloor +1$),
\item[iv)] $\{1,0,0\} \cup (\{0\} \times K^* \times \{0\}) \cup Y$ (with $Y \subseteq \{0\} \times \{0\} \times F^*$, $|Y|=\lfloor \frac{|F^*|}{2} \rfloor +1$),
\end{enumerate}
hence $\gamma^o(\Gamma( \mathbb{Z}_2 \times K \times F)) \leq 1 + \min A$.\\
Now, if $S$ is a $\gamma^o-$alliance and there are vertices $(0,k,0)$, $(0,0,f) \notin S$ with $k \in K^*$ y $f \in F^*$, then $|S| \geq 1 + |K^*| + |F^*| \geq 1 + \min A$ and we are done.
If $\{0\} \times K^* \times \{0\} \subseteq  S $ y $(1,k,0) \notin S$ for some $k \in K^*$ then $|S| \geq 2 + |K^*| + \left \lfloor \frac{|F^*|}{2} \right \rfloor \geq 1+ \min A$.\\
Similarly, if $\{0\} \times \{0\} \times F^* \subseteq S$ and some $(1,0,f) \notin S$.
\end{proof}

\section{Rings with small global offensive alliance number}\label{sec3}

\begin{Thm}
Let $R$ be a finite commutative ring, then $\gamma^o(\Gamma(R))=1$ if and only if $R$ is isomorphic to any of the following rings $\mathbb{Z}_4$, $\mathbb{Z}_2[X]/(X^2)$, $\mathbb{Z}_9$, $\mathbb{Z}_3[X]/(X^2)$, $\mathbb{Z}_8$, $\mathbb{Z}_2[X]/(X^3), \mathbb{Z}_2 \times F$ or $\mathbb{Z}_4[X]/(2X, X^2 -2)$, where $F$ is a field.
\end{Thm}
\begin{proof}
By \cite{RodSigFav} $\Gamma(R)$ is a vertex o a star graph. Now, if  $|\Gamma(R)|< 4$ then $R\cong \mathbb{Z}_6,\mathbb{Z}_8, \mathbb{Z}_2[X]/(X^3)$ or $\mathbb{Z}_4[X]/(2X, X^2 -2)$ (ver \cite{AndLiv} ). If  $|\Gamma(R)| \geq 4$ by \cite[Teorema 2.13]{AndLiv} $R\cong \mathbb{Z}\times F$ .\\
For the converse just verify each ring.
\end{proof}

 A graph $\Gamma(V,E)$ is a $4$\emph{-book} if $V$ is the union of subsets $X_1,X_2,\cdots,X_r$ such that $3\leq |X_i|\leq 4$ and
\begin{itemize}
\item there exists  $v_1, v_2 \in V$, a pair of distinct vertices such that if $ i \neq j$ then $X_i \cap X_j = \{v_1, v_2 \}$,
\item for $i \neq j$, there are no edges connecting vertices of $X_i \setminus \{v_1, v_2 \}$ with vertices of $X_j \setminus \{ v_1, v_2\}$,
\item if $y \in V \setminus \{ v_1, v_2\}$, then  is adjacent to both, $v_1$ and $v_2$.
\end{itemize}
Each one of these $r$ induced subgraphs, $\langle X_i \rangle$, is called $4$\emph{-book page} $\Gamma$.


\begin{Thm}
Let $R$ be a finite commutative ring, then
$\gamma^o(\Gamma(R))=2$ if and only if $R$ is isomorphic to one of the following rings $\mathbb{Z}_3 \times k$ (where $k$ is a field), $\mathbb{Z}_2 \times \mathbb{Z}_4$, $\mathbb{Z}_2 \times \mathbb{Z}_2[X]/(X^2)$, $\mathbb{Z}_{16}$, $\mathbb{Z}_2[X]/(X^4)$, $\mathbb{Z}_4[X]/(2X, X^3-2)$, $\mathbb{Z}_4[X]/(X^2-2)$, $\mathbb{Z}_4[X]/(X^2+2X+2)$, $\mathbb{F}_4[X]/(X^2)$, $\mathbb{Z}_4[X]/(X^2+X+1)$,   $\mathbb{Z}_2[X,Y]/(X,Y)^2$, $\mathbb{Z}_4[X]/(2,X)^2$, $\mathbb{Z}_{27}$, $\mathbb{Z}_3[X]/(X^3)$, $\mathbb{Z}_9[X]/(X^2-3,3X)$, $\mathbb{Z}_9[X]/(X^2-6,3X)$, $\mathbb{Z}_{25}$, $\mathbb{Z}_5[X]/(X^2)$.

\end{Thm}
\begin{proof}
By \cite{RodSigFav}, $\gamma^o(\Gamma(R))=2$ if and only if $\Gamma(R)$ is a $4-$book graph. But such graphs are planar graphs, in \cite{Smith} is shown that there are only $44$ types of rings whose zero-divisor graph is planar, of which the above ones are the only that meet being $4-$book graphs.

\end{proof}

\end{document}